\documentclass[reqno,12pt]{amsart}
\usepackage{amssymb}
\usepackage{amsfonts}
\usepackage{amsmath}
\usepackage{graphicx}
\usepackage{amscd,amsthm}

\newtheorem{theorem}{Theorem}
\theoremstyle{plain}
\newtheorem{acknowledgement}{Acknowledgement}

\newtheorem{corollary}{Corollary}

\newtheorem{definition}{Definition}

\numberwithin{equation}{section}

\input{tcilatex}

\begin{document}
\author{}
\title{}
\maketitle

\begin{center}
\thispagestyle{empty} \pagestyle{myheadings} 
\markboth{\bf Yilmaz Simsek
}{\bf Construction a new generating function of Bernstein type polynomials}

\textbf{{\Large Construction a \textbf{new }generating function of Bernstein 
\textbf{\textbf{type}} polynomials}}

\bigskip

\textbf{Yilmaz Simsek}\\[0pt]

\medskip

Department of Mathematics, Faculty of Art and Science\\[0pt]

University of Akdeniz\\[0pt]

TR-07058 Antalya, Turkey \\[0pt]

E-mail\textbf{: ysimsek@akdeniz.edu.tr}\\[0pt]

\bigskip \textit{Dedicated to Professor H. M. Srivastava on the occasion of
his seventieth birth anniversary}

\textbf{{\large {Abstract}}}\medskip
\end{center}

\begin{quotation}
Main purpose of this paper is to reconstruct generating function of the
Bernstein type polynomials. Some properties this generating functions are
given. By applying this generating function, not only derivative of these
polynomials but also recurrence relations of these polynomials are found.
Interpolation function of these polynomials is also constructed via Mellin
Transformation. This function interpolates these polynomials at negative
integers which are given explicitly. Moreover, relations between these
polynomials, the generalized Stirling numbers, and Bernoulli polynomials of
higher order are given. Furthermore some applications associated with B%
\'{}%
ezier curve are given.
\end{quotation}

\bigskip


\bigskip


\section{Introduction, Definitions and Preliminaries}

The Bernstein polynomials, recently, have been defined by many different
ways, for examples in $q$-series, by complex function and many algorithms.
These polynomials are used not only approximations of functions in various,
but also in the other fields such as smoothing in statistics, numerical
analysis, the solution of the differential equations, and constructing B%
\'{}%
ezier curve and in Computer Aided Design cf. (\cite{bernstein}, \cite%
{phillips-2}, \cite{Goldman}, \cite{goldman2}, \cite{SOstrovskaAM}, \cite%
{Simsek Acikgoz}, \cite{AcikgozSerkan}), and see also the references cited
in each of these earlier works.

By the same motivation of Ozden' \cite{Ozden} paper, which is related to the
unification of the Bernoulli, Euler and Genocchi polynomials, we, in this
paper, construct a generating function of the Bernstein polynomials which
unify generating function in \cite{Simsek Acikgoz}, \cite{AcikgozSerkan}.

\section{Construction generating functions of Bernstein type polynomials}

In this section we unify generating function of the Bernstein polynomials.
We define%
\begin{equation*}
\mathcal{F}(t,b,s:x)=\frac{2^{b}x^{bs}\left( \frac{t}{2}\right)
^{bs}e^{t(1-x)}}{\left( bs\right) !}
\end{equation*}%
where $b,s\in \mathbb{Z}^{+}:=\{1,2,3,\cdots \}$, $t\in \mathbb{C}$ and $%
x\in \left[ 0,1\right] $. This function is generating function of the
polynomials $\mathfrak{S}_{n}(bs,x)$:%
\begin{equation}
\mathcal{F}(t,b,s:x)=\sum_{n=0}^{\infty }\mathfrak{S}_{n}(bs,x)\frac{t^{n}}{%
n!},  \label{s1}
\end{equation}%
where $\mathfrak{S}_{0}(bs,x)=\cdots =\mathfrak{S}_{bs-1}(bs,x)=0$.

\textbf{Remark 1.} If we set $s=1$ in (\ref{s1}), we obtain%
\begin{equation*}
\frac{\left( xt\right) ^{b}e^{t(1-x)}}{b!}=\sum_{n=0}^{\infty }B_{n}(b,x)%
\frac{t^{n}}{n!},
\end{equation*}%
and $\mathfrak{S}_{n}(b,x)=B_{n}(b,x)$, which denotes the Bernstein
polynomials cf. (\cite{bernstein}, \cite{Goldman}, \cite{goldman2}, \cite%
{phillips-2}, \cite{Simsek Acikgoz}, \cite{AcikgozSerkan}).

By using Taylor expansion of $e^{t}$ in (\ref{s1}), we arrive at the
following theorem:

\begin{theorem}
\label{Teorem-1}Let $x,y\in \lbrack 0,1]$. Let $b,$ $n$ and $s\ $be
nonnegative integers. If $n\geq bs$, then we have%
\begin{equation*}
\mathfrak{S}_{n}(bs,x)=\left( 
\begin{array}{c}
n \\ 
bs%
\end{array}%
\right) \frac{x^{bs}(1-x)^{n-bs}}{2^{b(s-1)}}.
\end{equation*}
\end{theorem}

\textbf{Remark 1.} Setting $s=1$ in Theorem \ref{Teorem-1}, one can see that
the polynomials%
\begin{equation*}
\mathfrak{S}_{n}(b,x)=\left( 
\begin{array}{c}
n \\ 
b%
\end{array}%
\right) x^{b}(1-x)^{n-b},
\end{equation*}%
which give us the Bernstein polynomials cf. (\cite{Simsek Acikgoz}, \cite%
{AcikgozSerkan}). Consequently, the polynomials $\mathfrak{S}_{n}(bs,x)$ are
unification of the Bernstein polynomials.

By using Theorem \ref{Teorem-1}, we easily obtain the following results.

\begin{corollary}
Let $b$, $n$ and $s$ be nonnegative integers with $n\geq bs$. Then we have%
\begin{equation*}
\left( 
\begin{array}{c}
n \\ 
bs%
\end{array}%
\right) \mathfrak{S}_{n-bs}(bs;x)=\left( 
\begin{array}{c}
n+bs \\ 
n%
\end{array}%
\right) \mathfrak{S}_{n}(bs;x).
\end{equation*}
\end{corollary}

Setting%
\begin{equation*}
\mathfrak{g}_{n}(bs,x)=2^{b(s-1)}\mathfrak{S}_{n}(bs,x),
\end{equation*}%
where, for $bs=j$,%
\begin{equation*}
\sum_{j=0}^{n}\mathfrak{g}_{n}(j,x)=1.
\end{equation*}%
Let $f$ be a continuous function on $\left[ 0,1\right] $. Then we define
unification Bernstein type operator as follows:%
\begin{equation}
\mathbb{S}_{n}\left( f(x)\right) =\sum_{j=0}^{n}f\left( \frac{j}{n}\right) 
\mathfrak{g}_{n}(j;x),  \label{r6}
\end{equation}%
where $x\in \lbrack 0,1]$, $n$ is positive integer.

Setting $f(x)=x$ in (\ref{r6}), then we have%
\begin{equation*}
\mathbb{S}_{n}\left( x\right) =\sum_{j=0}^{n}\frac{j}{n}\left( 
\begin{array}{c}
n \\ 
j%
\end{array}%
\right) x^{j}(1-x)^{n-j}.
\end{equation*}%
From the above, we get%
\begin{equation*}
\mathbb{S}_{n}\left( x\right) =x\sum_{j=0}^{n}\mathfrak{g}_{n-1}(j-1,x).
\end{equation*}

\section{Fundamental relations of the polynomials $\mathfrak{S}_{n}(bs,x)$}

By using generating function of $\mathfrak{S}_{n}(bs,x)$, in this section we
give derivative of $\mathfrak{S}_{n}(bs,x)$ and recurrence relation of $%
\mathfrak{S}_{n}(bs,x)$.

\begin{theorem}
Let $x\in \lbrack 0,1]$. Let $b$, $n$ and $s$ be nonnegative integers with $%
n\geq bs$. Then we have%
\begin{equation}
\frac{d}{dx}\mathfrak{S}_{n}(bs,x)=n\left( \mathfrak{S}_{n-1}(bs-1,x)-%
\mathfrak{S}_{n}(bs,x)\right) .  \label{r2}
\end{equation}
\end{theorem}

\begin{proof}
By using the partial derivative of a function in (\ref{s1}) with respect to
the variable $x$, we have%
\begin{equation*}
\sum_{n=0}^{\infty }\frac{\partial }{\partial x}\left( \mathfrak{S}%
_{n}(bs,x)\right) \frac{t^{n}}{n!}=t\sum_{n=0}^{\infty }\mathfrak{S}%
_{n}(bs-1,x)\frac{t^{n}}{n!}-t\sum_{n=0}^{\infty }\mathfrak{S}_{n}(bs,x)%
\frac{t^{n}}{n!}.
\end{equation*}%
From the above, we obtain%
\begin{equation*}
\sum_{n=0}^{\infty }\left( \frac{d}{dx}\mathfrak{S}_{n}(bs,x)\right) \frac{%
t^{n}}{n!}=\sum_{n=0}^{\infty }n\mathfrak{S}_{n-1}(bs-1,x)\frac{t^{n}}{n!}%
-\sum_{n=0}^{\infty }n\mathfrak{S}_{n-1}(bs,x)\frac{t^{n}}{n!}.
\end{equation*}
\end{proof}

By using the partial derivative of a function in (\ref{s1}) with respect to
the variable $t$, we arrive at the following theorem:

\begin{theorem}
Let $x\in \lbrack 0,1]$. Let $b$, $n$ and $s$ be nonnegative integers with $%
n\geq bs$. Then we have%
\begin{equation}
\mathfrak{S}_{n}(bs,x)=x\mathfrak{S}_{n-1}(bs-1,x)+(1-x)\mathfrak{S}%
_{n-1}(bs,x).  \label{r3}
\end{equation}
\end{theorem}

\textbf{Remark 3.} If setting $s=1$, then (\ref{r3}) reduces to a recursive
relation of the Bernstein polynomials%
\begin{equation*}
B_{n}(b,x)=(1-x)B_{n-1}(b,x)+xB_{n-1}(b-1,x)
\end{equation*}%
and (\ref{r2}) reduces to derivative of the Bernstein polynomials%
\begin{equation*}
\frac{d}{dx}B_{n}(j,x)=n\left( B_{n-1}(j-1,x)-B_{n-1}(j,x)\right) ,
\end{equation*}%
respectively.

By the \textit{umbral calculus }convention in (\ref{s1}), we get%
\begin{equation*}
\frac{2^{b}x^{bs}\left( \frac{t}{2}\right) ^{bs}}{\left( bs\right) !}%
=e^{\left( \mathfrak{S}(bs,x)-(1-x)\right) t},
\end{equation*}%
where $\mathfrak{S}^{n}(bs;x)$ is replaced by $\mathfrak{S}_{n}(bs;x)$.
After some elementary calculation, we arrive at the following theorem.

\begin{theorem}
If $n=bs$, then we have%
\begin{equation*}
2^{b(1-s)}x^{bs}=\sum_{j=0}^{bs}\left( 
\begin{array}{c}
bs \\ 
j%
\end{array}%
\right) (-1)^{bs-j}\left( 1-x\right) ^{bs-j}\mathfrak{S}_{j}(bs,x).
\end{equation*}%
If $n>bs$, then we have%
\begin{equation*}
\sum_{j=bs+1}^{n}\left( 
\begin{array}{c}
n \\ 
j%
\end{array}%
\right) (-1)^{n-j}\left( 1-x\right) ^{n-j}\mathfrak{S}_{j}(bs,x)=0.
\end{equation*}
\end{theorem}

Relations between the polynomials the polynomial $\mathfrak{S}_{n}(bs,x)$,
Bernoulli polynomial of higher order and\ Stirling numbers of second kind is
given by the following theorem:

\begin{theorem}
Let $b$, $n$ and $s$ be nonnegative integers with $n\geq bs$. Then we have%
\begin{equation*}
\mathfrak{S}_{n}(bs,x)=2^{b(1-s)}x^{bs}\sum_{j=0}^{n}\left( 
\begin{array}{c}
n \\ 
j%
\end{array}%
\right) S(j,bs)B_{n-j}^{(bs)}(1-x),
\end{equation*}%
where $B_{n}^{(v)}(x)$ and $S(n,j)$ denote Bernoulli polynomial of higher
order and\ Stirling numbers of second kind, which are given by means of the
following generating function, respectively%
\begin{equation*}
\frac{t^{v}e^{xt}}{\left( e^{t}-1\right) ^{v}}=\sum_{n=0}^{\infty
}B_{n}^{(v)}(x)\frac{t^{n}}{n!},\text{ }(\left\vert t\right\vert <2\pi )
\end{equation*}%
and%
\begin{equation*}
(-1)^{v}\frac{\left( 1-e^{t}\right) ^{v}}{v!}=\sum_{n=0}^{\infty }S(n,v)%
\frac{t^{n}}{n!}.
\end{equation*}
\end{theorem}

\begin{proof}
By (\ref{s1}), we have%
\begin{equation*}
2^{b(1-s)}x^{bs}\left( \frac{(-1)^{bs}(e^{t}-1)^{bs}}{\left( bs\right) !}%
\right) \left( \frac{t^{bs}e^{(1-x)t}}{\left( e^{t}-1\right) ^{bs}}\right)
=\sum_{n=0}^{\infty }\mathfrak{S}_{n}(bs,x)\frac{t^{n}}{n!}.
\end{equation*}%
From the above, we have%
\begin{equation*}
\sum_{n=0}^{\infty }\mathfrak{S}_{n}(bs,x)\frac{t^{n}}{n!}%
=2^{b(1-s)}x^{bs}\left( \sum_{n=0}^{\infty }B_{n}^{(bs)}(1-x)\frac{t^{n}}{n!}%
\right) \left( \sum_{n=0}^{\infty }S(n,k)\frac{t^{n}}{n!}\right) .
\end{equation*}%
By Cauchy product in the above, after some calculation, we find the desired
result.
\end{proof}

By using same method of Lopez and Temme' \cite{LopezTemme}, we give contour
integral representation of $\mathfrak{S}_{n}(bs,x)$ as follows:%
\begin{equation*}
\mathfrak{S}_{n}(bs,x)=\frac{\Gamma (m+1)}{\Gamma (k+1)}\frac{1}{2\pi i}%
\int_{\mathcal{C}}\mathcal{F}(t,b,s:x)\frac{dz}{z^{m+1}},
\end{equation*}%
where $\mathcal{C}$ is a circle around the origin and the integration is in
positive direction.

\section{Interpolation Function of the polynomials $\mathfrak{S}_{n}(bs,x)$}

In this section, we construct meromorphic function. This function
interpolates $\mathfrak{S}_{n}(bs;x)$ at negative integers. These values are
given explicitly in Theorem \ref{TheoremNew1}.

For $z\in \mathbb{C}$, by applying the Mellin transformation to (\ref{s1}),
we obtain%
\begin{equation*}
\mathfrak{B}(z,bs;x)=\frac{1}{\Gamma (z)}\int_{0}^{\infty }t^{z-1}\mathcal{F}%
(-t,b,s:x)dt,
\end{equation*}%
where $\Gamma (z)$\ is Euler gamma function. From the above, we define the
following interpolation function.

\begin{definition}
\label{DefiniNew1}Let $z\in \mathbb{C}$ with $\Re (z)>0$ and $x\neq 1$. Let $%
b$ and $s$ be nonnegative integers. Then we define%
\begin{equation}
\mathfrak{B}(z,bs;x)=(-1)^{bs}\frac{\Gamma (z+bs)}{\Gamma (bs+1)\Gamma (z)}%
\frac{2^{b(1-s)}x^{bs}}{\left( 1-x\right) ^{z+bs}},  \label{12ef}
\end{equation}
\end{definition}

\textbf{Remark 4. }By the well-known identity $\Gamma (bs+1)=bs\Gamma (bs)$,
for $\Re (z)>0$ we have%
\begin{equation*}
\mathfrak{B}(z,k;x)=\frac{(-1)^{bs}2^{b(1-s)}x^{bs}}{bsB(z,k)\left(
1-x\right) ^{z+bs}},
\end{equation*}%
where $B(z,k)$ denotes the beta function. Observe that if $x=1$, then%
\begin{equation*}
\mathfrak{B}(z,bs,1)=\infty .
\end{equation*}

\begin{theorem}
\label{TheoremNew1}Let $b$, $n$ and $s$ be nonnegative integers with $n\geq
bs$ and $x\in \lbrack 0,1]$. Then we have%
\begin{equation*}
\mathfrak{B}(-n,bs;x)=\mathfrak{S}_{n}(bs,x).
\end{equation*}
\end{theorem}

\begin{proof}
Let $n$ and $b$, and $s$\ be positive integers with $bs\leq n$. $\Gamma (z)$
has simple poles at $z=-n=0,-1,-2,-3,\cdots $. The residue of $\Gamma (z)$ is%
\begin{equation*}
Res(\Gamma (z),-n)=\frac{(-1)^{n}}{n!}.
\end{equation*}%
Taking $z\rightarrow -n$ into (\ref{12ef}) and using the above relations,
the desired result can be obtained.
\end{proof}

Observe that if \ we set $s=1$ in Theorem \ref{TheoremNew1}, we arrive at%
\begin{equation*}
\mathfrak{B}(-n,b;x)=B_{n}(b,x).
\end{equation*}

\section{Further Remarks on B%
\'{}%
ezier curves}

The Bernstein polynomials are used to construct B%
\'{}%
ezier curves. B%
\'{}%
ezier was an engineer with the Renault car company and set out in the early
1960's to develop a curve formulation which would lend itself to shape
design. Engineers may find it most understandable to think of B%
\'{}%
ezier curves in terms of the center of mass of a set of point masses cf. 
\cite{Setberg}, for example, consider the four masses $m_{0}$, $m_{1}$, $%
m_{2}$, and $m_{3}$ located at points $P_{0}$, $P_{1}$, $P_{2}$, $P_{3}$.
The center of mass of these four point masses is given by the equation%
\begin{equation*}
P=\frac{m_{0}P_{0}+m_{1}P_{1}+m_{2}P_{2}+m_{3}P_{3}}{m_{0}+m_{1}+m_{2}+m_{3}}%
.
\end{equation*}%
Next, imagine that instead of being fixed, constant values, each mass varies
as a function of some parameter $x$. In specific case, let $m_{0}=(1-x)^{3}$%
, $m_{1}=3t(1-x)^{2}$, $m_{2}=3t^{2}(1-x)$ and $m_{3}=x^{3}$. The values of
these masses are a function of $x$. For each value of $x$, the masses assume
different weights and their center of mass changes continuously. As $x$
varies between $0$ and $1$, a curve is swept out by the center of masses.
This curve is a cubic B%
\'{}%
ezier curve. For any value of $x$, this B%
\'{}%
ezier curve is%
\begin{equation*}
P=m_{0}P_{0}+m_{1}P_{1}+m_{2}P_{2}+m_{3}P_{3},
\end{equation*}%
where $m_{0}+m_{1}+m_{2}+m_{3}\equiv 1$. These variable masses $m_{i}$ are
normally called \textit{blending functions} and their locations $P_{i}$ are
known as\textit{\ control points} or B%
\'{}%
ezier points. The blending functions, in the case of B%
\'{}%
ezier curves, are known as \textit{Bernstein polynomials. This} curve is
used in computer graphics and related fields and also in the time domain,
particularly in animation and interface design cf. (\cite{Goldman}, \cite%
{goldman2}, \cite{Setberg}).

The B%
\'{}%
ezier curve of degree $n$ can be generalized as follows. Given points $P_{0}$%
, $P_{1}$, $P_{2}$,$\cdots $, $P_{n}$ the B%
\'{}%
ezier curve is%
\begin{equation}
B(x)=\sum_{k=0}^{n}P_{k}B_{n}(k,x),  \label{r7}
\end{equation}%
where $x\in \lbrack 0,1]$ and $B_{n}(k,t)$ denotes Bernstein polynomials.

We now unify the B%
\'{}%
ezier curve in (\ref{r7}) by\ the polynomials $\mathfrak{g}_{n}(bs,x)$ as
follows%
\begin{equation*}
\mathbb{B}_{n}(x,y)=\sum_{k=0}^{n}P_{k}\mathfrak{g}_{n}(k;x),
\end{equation*}%
\ with control points $P_{k}$\textbf{.}

\begin{acknowledgement}
The present investigation was supported by the \textit{Scientific Research
Project Administration of Akdeniz University}.
\end{acknowledgement}

\end{document}